\documentclass[12pt]{amsart}

\usepackage{ucs}

\usepackage{amssymb}
\usepackage{amsthm}
\usepackage{amsmath}
\usepackage{latexsym}
\usepackage[cp1251]{inputenc}
\usepackage{graphicx}
\usepackage{wrapfig}
\usepackage{caption}
\usepackage{subcaption}
\usepackage{indentfirst}
\usepackage[left=2.6cm,right=2.6cm,top=2.95cm,bottom=2.95cm,bindingoffset=0cm]{geometry}
\usepackage{enumerate}

\DeclareMathOperator{\aut}{Aut}

\DeclareMathOperator{\cay}{Cay}
\DeclareMathOperator{\cyc}{Cyc}

\DeclareMathOperator{\id}{id}

\DeclareMathOperator{\iso}{Iso}

\DeclareMathOperator{\orb}{Orb}

\DeclareMathOperator{\rk}{rk}

\DeclareMathOperator{\Span}{Span}

\DeclareMathOperator{\syl}{Syl}
\DeclareMathOperator{\sym}{Sym}
\DeclareMathOperator{\rad}{rad}
\DeclareMathOperator{\Sup}{Sup}
\DeclareMathOperator{\DCI}{DCI}
\DeclareMathOperator{\CI}{CI}
\def\r{\mathrm{right}}
\DeclareMathOperator{\CA}{C1}
\DeclareMathOperator{\CAA}{C2}
\DeclareMathOperator{\CAAA}{C3}

\makeatletter 
\def\@seccntformat#1{\csname the#1\endcsname. } 
\def\@biblabel#1{#1.}

\makeatother

\title{The Cayley isomorphism property for the group $C_4\times C_p^2$}

\author{Grigory Ryabov}
\address{Sobolev Institute of Mathematics, Novosibirsk, Russia}
\address{Novosibirsk State University, Novosibirsk, Russia}
\email{gric2ryabov@gmail.com}
\thanks{The work is supported by the Russian Foundation for Basic Research (project 18-31-00051)}

\date{}

\newtheorem{prop}{Proposition}[section]
\newtheorem{theo}{Theorem}[section]

\newtheorem{lemm}[prop]{Lemma}

\theoremstyle{definition}

\begin{document}

\vspace{\baselineskip}
\vspace{\baselineskip}

\vspace{\baselineskip}

\vspace{\baselineskip}

\begin{abstract}
A finite group $G$ is called a \emph{$\DCI$}-group if two Cayley digraphs over $G$ are isomorphic if and only if their connection sets are conjugate by a group automorphism. We prove that the group $C_4\times C_p^2$, where $p$ is a prime, is a $\DCI$-group if and only if $p\neq 2$. 
\end{abstract}

\maketitle

\section{Introduction}

Let $G$ be a finite group. Given $S\subseteq G$, the \emph{Cayley digraph} $\cay(G,S)$ is defined to be the digraph with the vertex set $G$ and the arc set $\{(g,sg): g\in G, s\in S\}$. Two Cayley digraphs $\cay(G,S)$ and $\cay(G,T)$ are called \emph{Cayley isomorphic} if there exists $\varphi\in\aut(G)$ such that $S^{\varphi}=T$. It is easy to see that in this case $\varphi$ is an isomorphism from $\cay(G,S)$ to $\cay(G,T)$. So if two Cayley digraphs are  Cayley isomorphic then they are isomorphic. However, the converse statement is not true in general (see~\cite{AlPar, ET}). A subset $S\subseteq G$ is called a \emph{$\CI$-subset} if for each $T\subseteq G$, $\cay(G,S)\cong\cay(G,T)$  implies that $\cay(G,S)$ and $\cay(G,T)$ are Cayley isomorphic. A finite group $G$ is called a \emph{$\DCI$-group} (\emph{$\CI$-group}, respectively) if each subset of $G$ (each inverse-closed subset of $G$, respectively) is a $\CI$-subset.

One of the motivations to study $\DCI$-groups comes from the Cayley graph isomorphism problem. Let $G$ be a $\DCI$-group. Then to determine whether two Cayley digraphs $\cay(G,S)$ and $\cay(G,T)$ are isomorphic, we only need to check whether there exists $\varphi\in \aut(G)$ such that $S^{\varphi}=T$. Usually, the latter is much easier. For more details we refer the readers to~\cite[Section~3]{Li} and~\cite[Section~5]{MP1}.

The Cayley representation problem of graphs (see~\cite{Li,NP,Ry}) which seems to be very hard in general becomes easier in case when the input group is a $\DCI$-group. Recall that a \emph{Cayley representation} of a graph $\Gamma$ over a group $G$ is defined to be an isomorphism from $\Gamma$ to a Cayley graph over $G$. Two Cayley representations of $\Gamma$ are called \emph{equivalent} if the images of $\Gamma$ under these representations are Cayley isomorphic. The Cayley representation problem can be formulated as follows: given a group $G$ and a graph $\Gamma$ find a full set of non-equivalent Cayley representations of $\Gamma$ over $G$. From the definition of a $\DCI$-group it follows that every graph has at most one Cayley representation over a $\DCI$-group up to equivalence. More information on the Cayley representation problem and its connection with $\DCI$-groups can be found in~\cite{Li,NP}.

The starting point of the investigation of $\DCI$-groups was \'Ad\'am's conjecture~\cite{Adam}. In our terms, he conjectured that every cyclic group is a $\DCI$-group. However, this was proved to be false and first examples of non-$\DCI$ cyclic groups were found by Elspas and Turner~\cite{ET}. In~\cite{BF} Babai and Frankl posed the problem of determining of all finite $\DCI$- and $\CI$-groups. Accumulating results of several mathematicians, Li, Lu, and P\'alfy reduced the candidates of $\DCI$- and $\CI$-groups to a restricted list~\cite{LiLuP}. The most of the results on $\DCI$- and $\CI$-groups can be found in the survey paper~\cite{Li}.

One of the main steps towards the classification of all $\DCI$-groups is to determine all abelian $\DCI$-groups. The complete classification of cyclic $\DCI$-groups was obtained by Muzychuk in~\cite{M1,M2}. He proved that a cyclic group of order~$n$ is a $\DCI$-group if and only if $n=k$ or $n=2k$, where $k$ is square-free.

The cyclic group of order~$n$ is denoted by $C_n$. The classes of all finite abelian groups whose every Sylow subgroup is elementary abelian and of all finite abelian groups whose every Sylow subgroup is elementary abelian or isomorphic to $C_4$ are denoted by $\mathcal{E}$ and $\mathcal{E}_c$ respectively. From~\cite[Theorem~1.1]{KM} it follows that every abelian $\DCI$-group belongs to $\mathcal{E}_c$. However, the classification of all abelian $\DCI$-groups is far from complete. All known noncyclic abelian $\DCI$-groups (see~\cite{AlN,CLi,Dob,God,FK,HM,KM,KR2,MS}) belong to $\mathcal{E}$. In this paper we find the first example of an abelian noncyclic $\DCI$-group from $\mathcal{E}_c\setminus \mathcal{E}$. More precisely, we prove the following statement.

\begin{theo}\label{main}
Let $p$ be a prime. Then the group $C_4\times C_p^2$ is a $\DCI$-group if and only if $p\neq 2$.
\end{theo}

To prove Theorem~\ref{main}, we use the $S$-ring approach which was suggested by Klin and P\"{o}schel~\cite{KP} and developed in some subsequent papers~\cite{HM,KM,KR1,KR2,MS}. An \emph{$S$-ring} over a group $G$ is a subring of the group ring $\mathbb{Z}G$ which is a free $\mathbb{Z}$-module spanned by a partition of $G$ closed under taking inverse and containing the identity element of $G$ as a class. The notion of an $S$-ring goes back to Schur~\cite{Schur} and Wielandt~\cite{Wi}. To check that a given group $G$ is a $\DCI$-group, it is sufficient to show that every $S$-ring from a certain family of $S$-rings over $G$ is a $\CI$-$S$-ring (see Section~4). 

We finish the introduction with a brief outline of the paper. Sections~2, 3, and~4 contain a necessary background of $S$-rings, especially isomorphisms of $S$-rings, schurian $S$-rings, and $\CI$-$S$-rings. Section~5 is devoted to the tensor and generalized wreath products of $S$-rings. In this section we prove sufficient conditions of $\CI$-property for the tensor and generalized wreath products over a group from $\mathcal{E}_c$ (Lemma~\ref{citens} and Lemma~\ref{cigwr}) that extend known sufficient conditions of $\CI$-property for these constructions over a group from $\mathcal{E}$ (see~\cite{KM,KR1}). Sections~6 and~7 are concerned with $S$-rings over cyclic groups and $p$-$S$-rings respectively. In Section~8 we describe certain families of $S$-rings over $C_4\times C_p^2$. In Section~9 we prove Theorem~\ref{main}.
\\
\\
\\
{\bf Notation.}

The group of all permutations of a set $\Omega$ is denoted by $\sym(\Omega)$.

The set of non-identity elements of a group $G$ is denoted by  $G^\#$.

Let $X\subseteq G$. The element $\sum_{x\in X} {x}$ of the group ring $\mathbb{Z}G$ is denoted by $\underline{X}$.

The set $\{x^{-1}:x\in X\}$ is denoted by $X^{-1}$.

If  $m\in \mathbb{Z}$ then the set $\{x^m: x \in X\}$ is denoted by $X^{(m)}$.

The subgroup of $G$ generated by $X$ is denoted by $\langle X\rangle$; we also set $\rad(X)=\{g\in G:\ gX=Xg=X\}$.

The projections of $X\subseteq A\times B$ to $A$ and $B$ are denoted by $X_A$ and $X_B$ respectively.

The order of $g\in G$ is denoted by $|g|$.

The Sylow $p$-subgroup of $G$ is denoted by $\syl_p(G)$.

Given a set $X\subseteq G$ the set $\{(g,xg): x\in  X, g\in G\}$ of arcs of the Cayley digraph $\cay(G,X)$ is denoted by $R(X)$.

The subgroup of $\sym(G)$ consisting of all right translations of $G$ is denoted by $G_{\r}$.

The set $\{K \leq \sym(G):~ K \geq G_{\r}\}$ is denoted by $\Sup(G_{\r})$.

For a set $\Delta\subseteq \sym(G)$ and a section $S=U/L$ of $G$ we set 
$$\Delta^S=\{f^S:~f\in \Delta,~S^f=S\},$$
where $S^f=S$ means that $f$ permutes the $L$-cosets in $U$ and $f^S$ denotes the bijection of $S$ induced by $f$.

If $K\subseteq \sym(\Omega)$ and $\alpha\in \Omega$ then the set $\{f\in K:~\alpha^f=\alpha\}$ is denoted by~$K_{\alpha}$. 

If $K\leq \sym(\Omega)$ then the set of all orbits of $K$ on $\Omega$ is denoted by $\orb(K,\Omega)$.

The cyclic group of order $n$ is denoted by  $C_n$.

The class of all finite abelian groups whose every Sylow subgroup is elementary abelian is denoted by $\mathcal{E}$.

The class of all finite abelian groups whose every Sylow subgroup is elementary abelian or isomorphic to $C_4$ is denoted by $\mathcal{E}_c$.

\section{$S$-rings}

In this section we provide a background of $S$-rings. The most of the material of this section is contained in~\cite{KR2, MP2}. For more information on $S$-rings we refer the readers also to~\cite[Section~2.4]{CP} and~\cite{MP1}. 

Let $G$ be a finite group and $\mathbb{Z}G$ the integer group ring. Denote the identity element of $G$ by $e$. A subring  $\mathcal{A}\leq\mathbb{Z} G$ is called an \emph{$S$-ring (a Schur ring)} over $G$ if there exists a partition $\mathcal{S}(\mathcal{A})$ of~$G$ such that:

$(1)$ $\{e\}\in\mathcal{S}(\mathcal{A})$,

$(2)$  if $X\in\mathcal{S}(\mathcal{A})$ then $X^{-1}\in\mathcal{S}(\mathcal{A})$,

$(3)$ $\mathcal{A}=\Span_{\mathbb{Z}}\{\underline{X}:\ X\in\mathcal{S}(\mathcal{A})\}$.

\noindent The elements of $\mathcal{S}(\mathcal{A})$ are called the \emph{basic sets} of  $\mathcal{A}$ and the number $\rk(\mathcal{A})=|\mathcal{S}(\mathcal{A})|$ is called the \emph{rank} of  $\mathcal{A}$. If $X,Y \in \mathcal{S}(\mathcal{A})$ then $XY \in \mathcal{S}(\mathcal{A})$ whenever $|X|=1$ or $|Y|=1$. If $X,Y,Z\in\mathcal{S}(\mathcal{A})$ then   the number of distinct representations of $z\in Z$ in the form $z=xy$ with $x\in X$ and $y\in Y$ is denoted by $c^Z_{X,Y}$. Note that if $X$ and $Y$ are basic sets of $\mathcal{A}$ then $\underline{X}~\underline{Y}=\sum_{Z\in \mathcal{S}(\mathcal{A})}c^Z_{X,Y}\underline{Z}$. Therefore the numbers  $c^Z_{X,Y}$ are structure constants of $\mathcal{A}$ with respect to the basis $\{\underline{X}:\ X\in\mathcal{S}(\mathcal{A})\}$.

Let $\mathcal{A}$ be an $S$-ring over a group $G$. A set $X \subseteq G$ is called an \emph{$\mathcal{A}$-set} if $\underline{X}\in \mathcal{A}$. A subgroup $H \leq G$ is called an \emph{$\mathcal{A}$-subgroup} if $H$ is an $\mathcal{A}$-set. Given $\mathcal{A}$-sets $X$ and $Y$, the set $XY$ is also an $\mathcal{A}$-set. From the definition it follows that the intersection of $\mathcal{A}$-subgroups is an $\mathcal{A}$-subgroup. One can verify that for each $\mathcal{A}$-set $X$ the groups  $\langle X \rangle$ and $\rad(X)$ are $\mathcal{A}$-subgroups. 

Let $L \unlhd U\leq G$. A section $U/L$ is called an \emph{$\mathcal{A}$-section} if $U$ and $L$ are $\mathcal{A}$-subgroups. If $S=U/L$ is an $\mathcal{A}$-section then the module
$$\mathcal{A}_S=Span_{\mathbb{Z}}\left\{\underline{X}^{\pi}:~X\in\mathcal{S}(\mathcal{A}),~X\subseteq U\right\},$$
where $\pi:U\rightarrow U/L$ is the canonical epimorphism, is an $S$-ring over $S$.

\begin{lemm}~\cite[Lemma~2.1]{EKP}\label{intersection}
Let $\mathcal{A}$ be an $S$-ring over a group $G$, $H$ an $\mathcal{A}$-subgroup of $G$, and $X \in \mathcal{S}(\mathcal{A})$. Then the number $|X\cap Hx|$ does not depend on $x\in X$.
\end{lemm}

\begin{lemm}~\cite[Theorem~2.6]{MP2} \label{separat}
Let $\mathcal{A}$ be an $S$-ring over a group $G$. Suppose that $X\in \mathcal{S}(\mathcal{A})$ and $H\leq G$ are such that
$$X\cap H \neq \varnothing,~X \setminus H \neq \varnothing,~\text{and}~\langle X\cap H\rangle \leq \rad(X\setminus H).$$
Then $X=\langle X \rangle \setminus \rad(X)$ and $\rad(X) \leq H$.
\end{lemm}

We say that a set $X\subseteq G$ is rational if $X^{(m)}=X$ for every $m\in \mathbb{Z}$ coprime to~$|G|$. The next two statements are known as the Schur theorems on multipliers (see \cite[Theorem~23.9]{Wi}).

\begin{lemm} \label{burn}
Let $\mathcal{A}$ be an $S$-ring over an abelian group  $G$. Then $X^{(m)}\in \mathcal{S}(\mathcal{A})$  for every  $X\in \mathcal{S}(\mathcal{A})$ and every  $m\in \mathbb{Z}$ coprime to $|G|$.
\end{lemm}

\begin{lemm} \label{sch}
Let $\mathcal{A}$ be an $S$-ring over an abelian group $G$, $p$  a prime divisor of $|G|$, and  $H=\{g\in G:g^p=e\}$. Then for every  $\mathcal{A}$-set $X$ the set $X^{[p]}=\{x^p:x\in~X,~|X\cap Hx|\not\equiv 0\mod p\}$ is an $\mathcal{A}$-set.
\end{lemm}

\section{Isomorphisms and schurity}

Let  $\mathcal{A}$  and $\mathcal{A}^{\prime}$ be $S$-rings over groups $G$  and $G^{\prime}$ respectively. The identities of $\mathcal{A}$ and $\mathcal{A}^{\prime}$ are denoted by $e$ and $e^{\prime}$ respectively. A bijection $f:G\rightarrow G^{\prime}$ is called an \emph{isomorphism} from $\mathcal{A}$ to $\mathcal{A}^{\prime}$ if
$$\{R(X)^f: X\in \mathcal{S}(\mathcal{A})\}=\{R(X^{\prime}): X^{\prime}\in \mathcal{S}(\mathcal{A}^{\prime})\},$$
where $R(X)^f=\{(g^f,~h^f):~(g,~h)\in R(X)\}$. If there exists an isomorphism from $\mathcal{A}$ to $\mathcal{A}^{\prime}$ then we say that $\mathcal{A}$ and $\mathcal{A}^{\prime}$ are \emph{isomorphic} and write $\mathcal{A}\cong \mathcal{A}^{\prime}$. If $f$ is an isomorphism from $\mathcal{A}$ to $\mathcal{A}^{\prime}$ and $e^f=e^{\prime}$ then $\mathcal{S}(\mathcal{A})^f=\mathcal{S}(\mathcal{A}^{\prime})$ and $c_{X,Y}^Z=c_{X^f,Y^f}^{Z^f}$ for all $X,Y,Z\in \mathcal{S}(\mathcal{A})$.

The group  of all isomorphisms from $\mathcal{A}$ onto itself has a normal subgroup
$$\{f\in \iso(\mathcal{A}): R(X)^f=R(X)~\text{for every}~X\in \mathcal{S}(\mathcal{A})\}$$
called the \emph{automorphism group} of $\mathcal{A}$ and denoted by $\aut(\mathcal{A})$. From the definition it follows that $G_{\r}\leq \aut(\mathcal{A})$. One can check that $X^f=X$ for every $X\in \mathcal{S}(\mathcal{A})$ and every $f\in \aut(\mathcal{A})_e$. If $S$ is an $\mathcal{A}$-section then $\aut(\mathcal{A})^S\leq \aut(\mathcal{A}_S)$. Denote the group $\aut(\mathcal{A})\cap \aut(G)$ by $\aut_G(\mathcal{A})$. If $S$ is an $\mathcal{A}$-section then $\aut_G(\mathcal{A})^S\leq \aut_S(\mathcal{A}_S)$. 

Let $K\in \Sup(G_{\r})$. Schur proved in~\cite{Schur} that the $\mathbb{Z}$-submodule
$$V(K,G)=\Span_{\mathbb{Z}}\{\underline{X}:~X\in \orb(K_e,G)\},$$
is an $S$-ring over $G$. An $S$-ring $\mathcal{A}$ over  $G$ is called \emph{schurian} if $\mathcal{A}=V(K,G)$ for some $K\in \Sup(G_{\r})$. If $\mathcal{A}=V(K,G)$ for some $K\in \Sup(G_{\r})$ and $S$ is an $\mathcal{A}$-section then $\mathcal{A}_S=V(K^S,G)$. So if $\mathcal{A}$ is schurian then $\mathcal{A}_S$ is also schurian for every $\mathcal{A}$-section $S$. One can verify that if $\mathcal{A}=V(K,G)$ for some $K\in \Sup(G_{\r})$ then $K\leq \aut(\mathcal{A})$. Given $K\in \Sup(G_{\r})$ put $K^{(2)}=\aut(V(K,G))$. Put  also 
$$\Sup_2(G_\r)=\{K \in \Sup(G_\r):~ K^{(2)}=K\}.$$

An $S$-ring $\mathcal{A}$ over a group $G$ is called \emph{cyclotomic} if there exists $K\leq\aut(G)$ such that $\mathcal{S}(\mathcal{A})=\orb(K,G)$. In this case we write $\mathcal{A}=\cyc(K,G)$. Note that if $\mathcal{A}=\cyc(K,G)$ then $\mathcal{A}=V(G_{\r}K,G)$. So every cyclotomic $S$-ring is schurian. If $\mathcal{A}=\cyc(K,G)$ for some $K\leq \aut(G)$ and $S$ is an $\mathcal{A}$-section then $\mathcal{A}_S=\cyc(K^S,G)$. Therefore if $\mathcal{A}$ is cyclotomic then $\mathcal{A}_S$ is also cyclotomic for every $\mathcal{A}$-section $S$. 

Two groups $K_1,K_2\leq \aut(G)$ are defined to be \emph{Cayley equivalent} if $\orb(K_1,G)=\orb(K_2,G)$. In this case we write $K_1\cong_{\cay} K_2$. If $\mathcal{A}=\cyc(K,G)$ for some $K\leq \aut(G)$ then $\aut_G(\mathcal{A})$ is the largest group which is Cayley equivalent to $K$. A cyclotomic $S$-ring $\mathcal{A}$ over $G$ is called \emph{Cayley minimal} if
$$\{K\leq \aut(G):~K\approx_{\cay} \aut_G(\mathcal{A})\}=\{\aut_G(\mathcal{A})\}.$$
It is easy to see that $\mathbb{Z}G$ is Cayley minimal.

\section{$\CI$-$S$-rings}

Let $G$ be a finite group and $\mathcal{A}$ an $S$-ring over $G$.  Put 
$$\iso(\mathcal{A})=\{f\in \sym(G):~\text{f is an isomorphism from}~\mathcal{A}~\text{onto}~\text{$S$-ring over}~G\}.$$
One can see that $\aut(\mathcal{A})\aut(G)\subseteq \iso(\mathcal{A})$. The $S$-ring $\mathcal{A}$ is defined to be a \emph{$\CI$-$S$-ring} if 
$$\aut(\mathcal{A})\aut(G)=\iso(\mathcal{A}).$$ 
This definition was suggested in~\cite{HM}. From~\cite[Theorem~2.6]{HM} it follows that $\mathcal{A}$ is $\CI$ if and only if  
$$\aut(\mathcal{A})_e\aut(G)=\iso(\mathcal{A})_e.~\eqno(1)$$
 It is easy to check that $\mathbb{Z}G$ and the $S$-ring of rank~$2$ over $G$ are $\CI$-$S$-rings.  

Let $K_1, K_2\in \Sup(G_\r)$ with $K_1 \leq K_2$. Then $K_1$ is called a \emph{$G$-complete  
subgroup} of $K_2$ if every regular subgroup of $K_2$ isomorphic to $G$ is conjugate in $K_2$ to some subgroup of $K_1$ (see \cite[Definition~2]{HM}). In this case we write $K_1 \preceq_G K_2$. The relation $\preceq_G$ is a partial order on $\Sup(G_\r)$. The set of the minimal elements of $\Sup_2(G_\r)$ with respect to $\preceq_G$ is denoted by $\Sup_2^{\min}(G_\r)$.

\begin{lemm}\cite[Lemma~3.3]{KR2}\label{cimin}
Let $G$ be a finite group. If $V(K, G)$ is a $\CI$-S-ring for every $K \in \Sup_2^{\rm min}(G_\r)$ then $G$ is a $\DCI$-group.
\end{lemm}

\section{Tensor and generalized wreath products}

Let  $\mathcal{A}_1$ and $\mathcal{A}_2$ be $S$-rings over groups $G_1$ and $G_2$ respectively. Then the set
$$\mathcal{S}=\mathcal{S}(\mathcal{A}_1)\times \mathcal{S}(\mathcal{A}_2)=\{X_1\times X_2:~X_1\in \mathcal{S}(\mathcal{A}_1),~X_2\in \mathcal{S}(\mathcal{A}_2)\}$$
forms a partition of  $G=G_1\times G_2$ that defines an  $S$-ring over $G$. This $S$-ring is called the  \emph{tensor product}  of $\mathcal{A}_1$ and $\mathcal{A}_2$ and denoted by $\mathcal{A}_1 \otimes \mathcal{A}_2$.

\begin{lemm}\cite[Lemma 2.3]{EKP}\label{tenspr}
Let $\mathcal{A}$ be an $S$-ring over an abelian group $G=G_1\times G_2$. Suppose that $G_1$ and $G_2$ are $\mathcal{A}$-subgroups. Then 

$(1)$ $X_{G_i}\in \mathcal{S}(\mathcal{A})$ for all $X\in \mathcal{S}(\mathcal{A})$ and $i=1,2;$

$(2)$ $\mathcal{A} \geq \mathcal{A}_{G_1}\otimes \mathcal{A}_{G_2}$, and the equality is attained whenever $\mathcal{A}_{G_i}=\mathbb{Z}G_i$ for some $i\in \{1,2\}$.
\end{lemm}

\begin{lemm}\label{citens}
Let $G_1,G_2\in \mathcal{E}_c$. Suppose that $\mathcal{A}_1$ and $\mathcal{A}_2$ are $\CI$-$S$-rings over $G_1$ and $G_2$ respectively. Then $\mathcal{A}=\mathcal{A}_1 \otimes \mathcal{A}_2$ is a $\CI$-$S$-ring over $G=G_1\times G_2$.
\end{lemm}

\begin{proof}
To prove the lemma it sufficient to check that Eq.~(1) holds for $\mathcal{A}$, i.e. to check that for each isomorphism $f$ from $\mathcal{A}$ to an $S$-ring $\mathcal{A}^{\prime}$ over $G$ with $e^f=e$ there exists $\varphi\in \aut(G)$ satisfying the following:
$$X^f=X^{\varphi}$$
for every $X\in \mathcal{S}(\mathcal{A})$. Since $f$ is an isomorphism, 
$$\mathcal{A}^{\prime}=\mathcal{A}_1^{\prime} \otimes \mathcal{A}_2^{\prime},$$ 
where $\mathcal{A}_1^{\prime}$ and $\mathcal{A}_2^{\prime}$ are $S$-rings over groups $G_1^{\prime}$ and $G_2^{\prime}$ respectively with $G=G_1^{\prime}\times G_2^{\prime}$. The isomorphism $f$ induces the isomorphisms $f_1$ and $f_2$ from $\mathcal{A}_1$ to $\mathcal{A}_1^{\prime}$ and from $\mathcal{A}_2$ to $\mathcal{A}_2^{\prime}$ respectively. Note that $|G_1^{\prime}|=|G_1^{f_1}|=|G_1|$ and $|G_2^{\prime}|=|G_2^{f_2}|=|G_2|$. We conclude that $G_1^{\prime}\cong G_1$ and $G_2^{\prime}\cong G_2$ because $G_1,G_2,G_1^{\prime},G_2^{\prime}\in \mathcal{E}_c$. Without loss of generality, we may assume that $G_1^{\prime}=G_1$ and $G_2^{\prime}=G_2$.

Since $\mathcal{A}_1$ and $\mathcal{A}_2$ are $\CI$-$S$-rings, Eq.~(1) implies that there exist $\varphi_1\in \aut(G_1)$ and $\varphi_2\in \aut(G_2)$ with
$$X_1^{\varphi_1}=X_1^{f_1}~\text{and}~X_2^{\varphi_2}=X_2^{f_2}~\eqno(2)$$
for every $X_1\in \mathcal{S}(\mathcal{A}_1)$ and every $X_2\in \mathcal{S}(\mathcal{A}_2)$. Let $\varphi \in \aut(G_1)\times \aut(G_2)\leq \aut(G)$ such that $\varphi^{G_1}=\varphi_1$ and $\varphi^{G_2}=\varphi_2$. Prove that $X^f=X^{\varphi}$ for $\varphi$ and every $X\in \mathcal{S}(\mathcal{A})$. If $X\in \mathcal{S}(\mathcal{A}_1) \cup \mathcal{S}(\mathcal{A}_2)$ then $X^f=X^{\varphi}$ by Eq.~(2). If $X$ lies outside $\mathcal{S}(\mathcal{A}_1) \cup \mathcal{S}(\mathcal{A}_2)$ then $X=X_1X_2$ for some $X_1\in \mathcal{S}(\mathcal{A}_1)$ and $X_2\in \mathcal{S}(\mathcal{A}_2)$. The straightforward computation shows that
$$X^{\varphi}=(X_1X_2)^{\varphi}=X_1^{\varphi}X_2^{\varphi}=X_1^{\varphi_1}X_2^{\varphi_2}=X_1^{f_1}X_2^{f_2}=X_1^fX_2^f=X^f.$$
Here the fourth equality holds by Eq.~(2) and the sixth equality holds due to the following observation: $X_1^fX_2^f\in \mathcal{S}(\mathcal{A}_1^{\prime})\otimes \mathcal{S}(\mathcal{A}_2^{\prime})=\mathcal{S}(\mathcal{A}^{\prime})$ and $X^f$ is the unique basic set of $\mathcal{A}^{\prime}$ such that $c_{X_1^f,X_2^f}^{X^f}=c_{X_1,X_2}^X=1$. The lemma is proved.
\end{proof}

Let $\mathcal{A}$ be an $S$-ring over a group $G$ and $S=U/L$ an $\mathcal{A}$-section of $G$. An $S$-ring~$\mathcal{A}$ is called the \emph{$S$-wreath product} or the \emph{generalized wreath product} of $\mathcal{A}_U$ and $\mathcal{A}_{G/L}$ if $L\trianglelefteq G$ and $L\leq\rad(X)$ for each basic set $X$ outside~$U$. In this case we write $\mathcal{A}=\mathcal{A}_U\wr_{S}\mathcal{A}_{G/L}$ and omit $S$ when $U=L$. The construction of the generalized wreath product of $S$-rings was introduced  in~\cite{EP1}. 

A sufficient condition of $\CI$-property for the generalized wreath product over a group from $\mathcal{E}$ was proved in~\cite{KR1}. However, an abelian $\DCI$-group can belong to $\mathcal{E}_c\setminus \mathcal{E}$. On the other hand, each abelian $\DCI$-group lies in $\mathcal{E}_c$ (see~\cite[Theorem~1.1]{KM}). The following statement extends the sufficient condition from~\cite{KR1} to all groups from $\mathcal{E}_c$.

\begin{lemm}\label{cigwr}
Let $G\in \mathcal{E}_c$, $\mathcal{A}$ an $S$-ring over $G$, and $S=U/L$ an $\mathcal{A}$-section of $G$. Suppose that $\mathcal{A}$ is the nontrivial $S$-wreath product and the $S$-rings $\mathcal{A}_U$ and $\mathcal{A}_{G/L}$ are $\CI$-$S$-rings. Then $\mathcal{A}$ is a $\CI$-$S$-ring whenever 
$$\aut_{S}(\mathcal{A}_{S})=\aut_U(\mathcal{A}_U)^{S}\aut_{G/L}(\mathcal{A}_{G/L})^{S}.$$ 
In particular,  $\mathcal{A}$ is a $\CI$-$S$-ring if $\aut_{S}(\mathcal{A}_{S})=\aut_U(\mathcal{A}_U)^{S}$ or $\aut_{S}(\mathcal{A}_{S})=\aut_{G/L}(\mathcal{A}_{G/L})^{S}$.
\end{lemm}

Before we start with the proof of Lemma~\ref{cigwr}, we prove an easy auxiliary lemma on elements of  groups from $\mathcal{E}_c$.

\begin{lemm}\label{orders}
Let $G\in \mathcal{E}_c$, $L\leq G$, and $x,y\in G\setminus L$ such that 

$(1)$ $|x|$ is an odd prime or $|x|=4$;

$(2)$ the orders of $Lx$ and $Ly$ (considered as elements of the group $G/L$) are equal.

\noindent Then there exist $y^{\prime}\in Ly$ with $|x|=|y|$.
\end{lemm}

\begin{proof}
Put $|x|=k$. The second condition of the lemma implies that $y^k=z$ for some $z\in L$. Suppose that $(k,m)>1$, where $m=|z|$. If $k$ is an odd prime then $|y|$ is divisible by~$k^2$; if $k=4$ then $|y|$ is divisible by~$8$. In both cases we obtain a contradiction with  $G\in \mathcal{E}_c$. So $(k,m)=1$ and hence there exist $l$ and $n$ with $ml+kn=1$. One can see that $z=z^{ml+kn}=z^{kn}$. Therefore $y^k=z^{kn}$. This yields that $(yz^{-n})^k=e$. Since $y\notin L$, the element $yz^{-n}$ is nontrivial. If $|yz^{-n}|=k$ then $y^{\prime}=yz^{-n}$ is the required element. 

Suppose that $|yz^{-n}|\neq k$. In view of the first condition of the lemma, this is possible only if $k=4$ and $|yz^{-n}|=2$. Due to the second condition of the lemma and the equality $|yz^{-n}|=2$, we conclude that $x^2\in L$. Clearly, $|x^2|=2$. The group $G$ contains the unique element of order~$2$ because $G\in \mathcal{E}_c$. So $yz^{-n}=x^2\in L$. This means that $y\in L$, a contradiction with the assumption of the lemma. The lemma is proved.
\end{proof}

\begin{proof}[Proof of Lemma~\ref{cigwr}]
If $G\in \mathcal{E}$ then the statement of the lemma follows from~\cite[Theorem~1]{KR1}. Further we assume that $G\in \mathcal{E}_c\setminus \mathcal{E}$, i.e. $\syl_2(G)\cong C_4$. Let $f$ be an isomorphism from $\mathcal{A}$ to an $S$-ring $\mathcal{A}^{\prime}$ over $G$ with $e^f=e$. The $S$-ring $\mathcal{A}^{\prime}$ is the $U^f/L^f$-wreath product by~\cite[Theorem~3.3,~1]{EP3}. Clearly, $|U^f|=|U|$ and $|L^f|=|L|$. Since $U,L,U^f,L^f\in \mathcal{E}_c$, we conclude that $U^f\cong U$ and $L^f \cong L$. Without loss of generality, we may assume that $U^f=U$ and $L^f=L$.

The isomorphism $f$ induces the isomorphisms $f_1$ and $f_2$ from $\mathcal{A}_U$ to $\mathcal{A}_U^{\prime}$ and from $\mathcal{A}_{G/L}$ to $\mathcal{A}_{G/L}^{\prime}$ respectively. Since $\mathcal{A}_U$ and $\mathcal{A}_{G/L}$ are $\CI$-$S$-rings, Eq.~(1) implies that there exist $\varphi_1\in \aut(U)$ and $\varphi_2\in \aut(G/L)$ with
$$X_1^{\varphi_1}=X_1^{f_1}~\text{and}~X_2^{\varphi_2}=X_2^{f_2}$$
for every $X_1\in \mathcal{S}(\mathcal{A}_U)$ and every $X_2\in \mathcal{S}(\mathcal{A}_{G/L})$. One can see that $X^{\varphi_1^S}=X^{\varphi_2^S}$ for every $X\in \mathcal{S}(\mathcal{A}_S)$ because $X^{f_1^S}=X^{f^S}=X^{f_2^S}$. So $\varphi_1^S(\varphi_2^S)^{-1}\in \aut_S(\mathcal{A}_S)$. By the condition of the lemma, there exist $\sigma_1\in \aut_U(\mathcal{A}_U)$ and $\sigma_2\in \aut_{G/L}(\mathcal{A}_{G/L})$ such that $\varphi_1^S(\varphi_2^S)^{-1}=\sigma_1^S\sigma_2^S$. Put
$$\psi_1=\sigma_1^{-1}\varphi_1~\text{and}~\psi_2=\sigma_2\varphi_2.$$
Note that
$$X_1^{\psi_1}=X_1^{f_1}~\text{and}~X_2^{\psi_2}=X_2^{f_2}$$
for every $X_1\in \mathcal{S}(\mathcal{A}_U)$ and every $X_2\in \mathcal{S}(\mathcal{A}_{G/L})$ because $\sigma_1\in \aut_U(\mathcal{A}_U)$ and $\sigma_2\in \aut_{G/L}(\mathcal{A}_{G/L})$. This implies also that $U^{\psi_1}=U$ and $L^{\psi_1}=L$. The straightforward check shows that
$$\psi_1^S=(\sigma_1^S)^{-1}\varphi_1^S=(\sigma_1^S)^{-1}\sigma_1^S\sigma_2^S\varphi_2^S=\sigma_2^S\varphi_2^S=\psi_2^S.~\eqno(3)$$

Denote the Sylow $2$-subgroup of $G$ and one of its generators by~$C$ and~$c$ respectively. Put $c_1=c^2$ and $C_1=\langle c_1\rangle$. If $\syl_2(U)$ is trivial or $\syl_2(U)=C$ then there exists $H\in \mathcal{E}_c$ such that $G=H\times U$. If $\syl_2(U)=C_1$ then there exists $H\in \mathcal{E}$ such that $G=C\times H \times V$, where $V\in \mathcal{E}$ satisfies the equality $U=C_1\times V$.

One can present $H$ in the form $H=\langle x_1 \rangle \times \ldots \langle x_s \rangle$, where $|x_i|$ is an odd prime or equal to~$4$ for every $i\in \{1,\ldots,s\}$. Let $(Lx_i)^{\psi_2}=Lx_i^{\prime}$, $i\in \{1,\ldots,s\}$. In view of Lemma~\ref{orders}, we may assume that 
$$|x_i|=|x_i^{\prime}|~\eqno(4)$$
for every $i\in \{1,\ldots,s\}$. 
Put $H^{\prime}=\langle x_1^{\prime},\ldots, x_s^{\prime}\rangle$. One can see that $H^{\prime}=\langle x_1^{\prime} \rangle \times \ldots \langle x_s^{\prime} \rangle$ and $H^{\prime}\cap L=\{e\}$ because $\psi_2$ is an isomorphism and $H\cap U=H\cap L=\{e\}$. Together with Eq.~(4), this yields that
$$|H^{\prime}|=|H|.~\eqno(5)$$

Note that 
$$H^{\prime}\cap U=\{e\}.~\eqno(6)$$ 
Indeed, if $g\in H^{\prime}\cap U$ then $(Lg)^{\psi_2^{-1}}\in S\cap H/L$ because $S^{\psi_2}=S^{f_2}=S$. So $g\in H^{\prime} \cap L=\{e\}$.

If $\syl_2(U)$ is trivial or $\syl_2(U)=C$ then from the definition of $H$, Eqs.~(5) and~(6) it follows that 
$$G=H^{\prime}\times U.$$
Therefore one can define $\alpha\in \aut(G)$ as follows:
$$\alpha^U=\psi_1,~x_i^{\alpha}=x_i^{\prime},~i\in\{1,\ldots,s\}.$$

Suppose that $\syl_2(U)=C_1$. Let $(Lc)^{\psi_2}=Lc^{\prime}$. Due to Lemma~\ref{orders}, we may assume that $|c|=|c^{\prime}|=4$. So $\langle c^{\prime} \rangle=C$. In view of Eq.~(6), we have $H^{\prime}\cap V=\{e\}$ and hence $\langle H^{\prime}, V \rangle=H^{\prime} \times V$. From Eq.~(5) it follows that $|H^{\prime}\times V|=|H\times V|=|G|/4$. Thus,
$$G=C^{\prime}\times H^{\prime}\times V.$$
Therefore one can define $\alpha\in \aut(G)$ as follows:
$$\alpha^V=\psi_1,c^{\alpha}=c^{\prime},~x_i^{\alpha}=x_i^{\prime},~i\in\{1,\ldots,s\}.$$ 
Note that $c_1^{\alpha}=c_1^{\psi_1}=c_1$ because $c_1$ is the unique element of order~$2$ in $G$. So $\alpha^U=\psi_1^U$.

Let us show that $\alpha$ satisfies the following:
$$X^{\alpha}=X^f$$
for every $X\in \mathcal{S}(\mathcal{A})$. From the definition of $\alpha$ it follows that $L^{\alpha}=L^{\psi_1}=L$, $U^{\alpha}=U^{\psi_1}=U$, and $(Lh)^{\psi_2}=Lh^{\alpha}$ for every $h\in H$. Also $(Lc)^{\psi_2}=Lc^{\alpha}$ in case when $\syl_2(U)=C_1$. Let $X\in \mathcal{S}(\mathcal{A})$.  If $X\subseteq U$ then $X^{\alpha}=X^{\psi_1}=X^{f_1}=X^f$. Suppose that $X$ lies outside $U$. Then $L\leq \rad(X)$ and hence $X$ can be presented in the form 
$$X=Lh_1u_1\cup \ldots \cup Lh_tu_t,$$
where $h_i\in H$ if $|\syl_2(U)|\in\{1,4\}$ and $h_i\in (C\times H)\setminus C_1$ otherwise, and $u_i\in U\setminus L$ for each $i\in\{1,\ldots,t\}$. 

The straightforward check implies that
\begin{eqnarray}
\nonumber X^f/L=(X/L)^{f_2}=(X/L)^{\psi_2}=\{Lh_1u_1,\ldots,Lh_tu_s\}^{\psi_2}=\\
\nonumber\{(Lh_1)^{\psi_2}(Lu_1)^{\psi_2},\ldots,(Lh_t)^{\psi_2}(Lu_t)^{\psi_2}\}=\{Lh_1^{\alpha}u_1^{\psi_1},\ldots,Lh_t^{\alpha}u_t^{\psi_1}\}=X^{\alpha}/L.
\end{eqnarray}
The fifth equality holds by Eq.~(3). Since $L\leq \rad(X^f)$ and $L\leq \rad(X^{\alpha})$, we conclude that $X^{\alpha}=X^f$. 

Thus, we proved that for every isomorphism $f$ from $\mathcal{A}$ to any $S$-ring over $G$ satisfying $e^f=e$ there exists $\alpha\in \aut(G)$ with $X^{\alpha}=X^f$ for every $X\in \mathcal{S}(\mathcal{A})$. So Eq.~(1) holds for $\mathcal{A}$ and hence $\mathcal{A}$ is a $\CI$-$S$-ring. The lemma is proved.
\end{proof}

The proofs of the next two lemmas are similar to the proofs of~\cite[Proposition~4.1]{KR1} and~\cite[Lemma~4.2]{KR2}. We provide these proofs for the convenience of the readers. 

\begin{lemm}\label{trivial}
In the conditions of Lemma~\ref{cigwr}, suppose that 
$$\mathcal{A}_S=\mathbb{Z}S.~\eqno(\CA)$$ 
Then $\mathcal{A}$ is a $\CI$-$S$-ring.
\end{lemm}

\begin{proof}
On the one hand, $\aut_U(\mathcal{A}_U)^S\leq \aut_S(\mathcal{A}_S)$. On the other hand, $\aut_S(\mathcal{A}_S)$ is trivial because $\mathcal{A}_S=\mathbb{Z}S$. So $\aut_U(\mathcal{A}_U)^S=\aut_S(\mathcal{A}_S)$ and $\mathcal{A}$ is a $\CI$-$S$-ring by Lemma~\ref{cigwr}. The lemma is proved.
\end{proof}

\begin{lemm}\label{cayleymin}
In the conditions of Lemma~\ref{cigwr}, suppose that 
$$\mathcal{A}_U~\text{or}~\mathcal{A}_{G/L}~\text{is cyclotomic,}~\mathcal{A}_S~\text{is Cayley minimal}.~\eqno(\CAA)$$
Then $\mathcal{A}$ is a $\CI$-$S$-ring.
\end{lemm}

\begin{proof}
Suppose that $\mathcal{A}_U$ is cyclotomic. Then $\mathcal{A}_S$ is also cyclotomic and $\aut_U(\mathcal{A}_U)^S\cong_{\cay}\aut_S(\mathcal{A}_S)$. This implies that $\aut_U(\mathcal{A}_U)^S=\aut_S(\mathcal{A}_S)$ because $\mathcal{A}_S$ is Cayley minimal. Therefore $\mathcal{A}$ is a $\CI$-$S$-ring by Lemma~\ref{cigwr}. The case when $\mathcal{A}_{G/L}$ is cyclotomic is similar. The lemma is proved.
\end{proof}

\begin{lemm}\label{restr}
In the conditions of Lemma~\ref{cigwr}, suppose that
$$U\cong C_{4p},~L\cong C_2,~\text{and}~\mathcal{A}_U=(\mathbb{Z}L\otimes \mathcal{A}_V)\wr_{(V\times L)/V}\mathbb{Z}(U/V),~\eqno(\CAAA)$$
where $p$ is a prime and $V$ is an $\mathcal{A}_U$-subgroup of order~$p$. Then $\mathcal{A}$ is a $\CI$-$S$-ring.
\end{lemm}

\begin{proof}
Let $c\in U$ with $|c|=4$, $C=\langle c \rangle$, $c_1=c^2$, and $\pi:G\rightarrow G/L$ the canonical epimorphism. Note that $\mathcal{A}_S=\mathcal{A}_{V^{\pi}}\wr \mathbb{Z}(U^{\pi}/V^{\pi})\cong \mathcal{A}_V \wr \mathbb{Z}C_2$. Let $\varphi\in \aut_S(\mathcal{A}_S)$. Then $(c^{\pi})^{\varphi}=c^{\pi}$ because $c^{\pi}$ is the unique element of order~$2$ in $S$ and $(V^{\pi})^{\varphi}=V^{\pi}$ because $V^{\pi}$ is the unique subgroup of $S$ of order~$p$. Clearly, $U=C\times V$. So there exists $\psi \in \aut(U)$ such that
$$C^{\psi}=C,~V^{\psi}=V,~\psi^C=\id_C,~\text{and}~x^{\psi\pi}=x^{\pi\varphi}~\text{for every}~x\in V.$$
By the definition, $\psi^S=\varphi$. Prove that $\psi\in \aut_U(\mathcal{A}_U)$. Let $X\in \mathcal{S}(\mathcal{A}_U)$. If $X\nsubseteq V\times L$ then $X=Vc^k$, where $k\in\{-1,1\}$. One can see that 
$$X^{\psi}=(Vc^k)^{\psi}=V^{\psi}(c^k)^{\psi}=Vc^k=X.$$ 
If $X=\{c_1\}$ then $X^{\psi}=X$ by the definition of $\psi$. If $X\subseteq V$ then the equality $X^{\psi}=X$ follows from the definition of $\psi$ and the facts that $\mathcal{A}_V \cong \mathcal{A}_{V^{\pi}}$ and $\varphi^{V^{\pi}}\in \aut_{V^{\pi}}(\mathcal{A}_{V^{\pi}})$. If $X\subseteq V\times L\setminus (V \cup L)$ then $X=X_1c_1$, where $X_1\in \mathcal{S}(\mathcal{A}_V)$. It follows that 
$$X^{\psi}=X_1^{\psi}(c_1)^{\psi}=X_1c_1=X.$$
Therefore $X^{\psi}=X$ for every $X\in \mathcal{S}(\mathcal{A}_U)$ and hence $\psi\in \aut_U(\mathcal{A}_U)$. This implies that $\aut_U(\mathcal{A}_U)^S\geq \aut_S(\mathcal{A}_S)$. On the other hand, $\aut_U(\mathcal{A}_U)^S\leq \aut_S(\mathcal{A}_S)$. Thus, $\aut_U(\mathcal{A}_U)^S=\aut_S(\mathcal{A}_S)$ and $\mathcal{A}$ is a $\CI$-$S$-ring by Lemma~\ref{cigwr}. The lemma is proved.
\end{proof}

\section{$S$-rings over cyclic groups}

In this section we provide the results on $S$-rings over cyclic groups.

\begin{lemm}\label{circ}
Let $\mathcal{A}$ be an $S$-ring over a cyclic group $G$. Then one of the following statements holds:

$(1)$ $\rk(\mathcal{A})=2$;

$(2)$ $\mathcal{A}$ is cyclotomic;

$(3)$ $\mathcal{A}$ is the tensor product of two $S$-rings over proper subgroups of $G$;

$(4)$ $\mathcal{A}$ is the nontrivial $S$-wreath product for some $\mathcal{A}$-section $S$.

\noindent~If in addition $|G|$ is prime then $\mathcal{A}$ is cyclotomic.
\end{lemm}

\begin{proof}
The first part of the lemma follows from~\cite[Theorems~4.1, 4.2]{EP2}. The second part of the lemma follows from the first part and the next observation: if $|G|$ is a prime and $\mathcal{A}$ is the $S$-ring of rank~$2$ over $G$ then $\mathcal{A}=\cyc(\aut(G),G)$. The lemma is proved.
\end{proof}

\begin{lemm}\label{circcaymin}
Let $\mathcal{A}$ be a cyclotomic $S$-ring over a cyclic group $G$. Then $\mathcal{A}$ is Cayley minimal.
\end{lemm}

\begin{proof}
Let $x$ be a generator of $G$ and $X\in \mathcal{S}(\mathcal{A})$ with $x\in X$. Clearly, if $\varphi\in \aut_G(\mathcal{A})$ and $x^{\varphi}=x$ then $\varphi$ is trivial. So $|\aut_G(\mathcal{A})|=|X|$. This implies that there are no proper subgroups $K$ of $\aut_G(\mathcal{A})$ with $\mathcal{A}=\cyc(K,G)$. Therefore $\mathcal{A}$ is Cayley mininmal. The lemma is proved.
\end{proof}

\section{$p$-$S$-rings}

Let $p$ be a prime number. An $S$-ring $\mathcal{A}$ over a $p$-group $G$ is called a \emph{$p$-S-ring} if every basic set of $\mathcal{A}$ has $p$-power size.

\begin{lemm}\label{minpring}
Let $G$ be an abelian group, $K \in \Sup_2^{\min}(G_\r)$, and $\mathcal{A}=V(K, G)$. Suppose that $H$ is an $\mathcal{A}$-subgroup of $G$ such that $G/H$ is a $p$-group for some prime $p$. Then $\mathcal{A}_{G/H}$ is a $p$-S-ring.
\end{lemm}

\begin{proof}
The statement of the lemma follows from~\cite[Lemma~5.2]{KM}.
\end{proof}

\begin{lemm}\label{p2}
Let $G\cong C_4$ or $G\cong C_p^2$, where $p$ is a prime. Then $\mathcal{A}=\mathbb{Z}G$ or $\mathcal{A}\cong \mathbb{Z}C_q\wr \mathbb{Z}C_q$, where $q\in\{2,p\}$. In both cases $\mathcal{A}$ is cyclotomic.
\end{lemm}

\begin{proof}
If $G\cong C_4$ then the statement of the lemma can be verified directly. If $G\cong C_p^2$ then the statement of the lemma follows from~\cite[Lemma~5.1]{KR1}. 
\end{proof}

\section{$S$-rings over $C_4\times C_p^2$}

Let $p$ be a prime number and $G\cong C\times H$, where $C\cong C_4$ and $H\cong C_p^2$. Denote a generator of $C$ by $c$. Put $c_1=c^2$ and $C_1=\langle c_1 \rangle$. These notations are valid until the end of the paper. The main goal of this section is to describe schurian $S$-rings over $G$ whose automorphism groups belong to $\Sup_2^{\rm min}(G_\r)$ in case when $p$ is odd. The main result of the section is given in the next proposition.

\begin{theo}\label{main2}
Let $p$ be an odd prime, $K \in \Sup_2^{\rm min}(G_\r)$, and $\mathcal{A}=V(K,G)$. Then one of the following statements holds:

$(1)$ $\rk(\mathcal{A})=2$;

$(2)$ $\mathcal{A}$ is the tensor product of two $S$-rings over proper subgroups of $G$;

$(3)$ $\mathcal{A}$ is the nontrivial $S=U/L$-wreath product satisfying one of the conditions $(\CA)$-$(\CAAA)$.

\end{theo}

We divide the proof of Theorem~\ref{main2} into cases depending on whether $C$ or $H$ is an $\mathcal{A}$-subgroup.

\begin{prop}\label{proof1}
In the conditions of Theorem~\ref{main2}, suppose that $C$ and $H$ are $\mathcal{A}$-subgroups. Then Statement~2 of Theorem~\ref{main2} holds for $\mathcal{A}$.
\end{prop}

\begin{proof}
The $S$-ring $\mathcal{A}_C$ is a $2$-$S$-ring and $\mathcal{A}_H$ is a $p$-$S$-ring by Lemma~\ref{minpring}. Lemma~\ref{p2} implies that each of the $S$-rings $\mathcal{A}_C$ and $\mathcal{A}_H$ coincides with the group ring or isomorphic to the wreath product of two group rings. If at least one of the $S$-rings $\mathcal{A}_C$ and $\mathcal{A}_H$ coincides with the group ring then $\mathcal{A}=\mathcal{A}_C \otimes \mathcal{A}_H$ by Statement~2 of Lemma~\ref{tenspr}. So we may assume that 
$$\mathcal{A}_C=\mathbb{Z}C_1\wr \mathbb{Z}(C/C_1)~\text{and}~\mathcal{A}_H=\mathbb{Z}A\wr \mathbb{Z}(H/A)~\eqno(7)$$ 
for some $\mathcal{A}_H$-subgroup $A$ of order~$p$. 

Let $X\in \mathcal{S}(\mathcal{A})$ outside $C$ and $H$. Statement~1 of Lemma~\ref{tenspr} implies that $X_C,X_H\in \mathcal{S}(\mathcal{A})$. If $|X_C|=1$ or $|X_H|=1$ then $X=X_C\times X_H$. If $|X_C|>1$ and $|X_H|>1$ then $X_C=\{c,c^{-1}\}$ and $|X_H|=p$ by Eq.~(7). The numbers $\lambda=|Cx\cap X|$ and $\mu=|Hx\cap X|$ does not depend on $x\in X$ by Lemma~\ref{intersection}. If $\lambda=2$ then $X=X_C\times X_H$. If $\lambda=1$ then $X=X_1c\cup X_2c^{-1}$, where $X_1,X_2\subseteq H$, $X_1\cap X_2=\varnothing$, and $|X_1|=|X_2|=\mu$. In this case we obtain that
$$2\mu=|X_1|+|X_2|=|X_H|=p.$$
However, this contradicts to the fact that $p$ is prime. Therefore, in any case $X=X_C \times X_H$. So $\mathcal{A}=\mathcal{A}_C \otimes \mathcal{A}_H$ and Statement~2 of Theorem~\ref{main2} holds for $\mathcal{A}$. The proposition is proved.
\end{proof}

\begin{prop}\label{proof2}
In the conditions of  Theorem~\ref{main2}, suppose that $C$ is an $\mathcal{A}$-subgroup whereas $H$ is not an $\mathcal{A}$-subgroup. Then Statement~3 of Theorem~\ref{main2} holds for $\mathcal{A}$.
\end{prop}

\begin{proof}
The $S$-ring $\mathcal{A}_{G/C}$ is a $p$-$S$-ring by Lemma~\ref{minpring}. So $\mathcal{A}_{G/C}=\mathbb{Z}(G/C)$ or $\mathcal{A}_{G/C}\cong \mathbb{Z}C_p\wr \mathbb{Z}C_p$. Denote the canonical epimorphism from $G$ to $G/C$ by~$\pi$. Let $X\in\mathcal{S}(\mathcal{A})$ outside $C$. Then $|X^{\pi}|\in\{1,p\}$.

Suppose that $|X^{\pi}|=1$. Then $U=\langle C, X\rangle$ is a cyclic group of order~$4p$. So one of the statements~2-4 of Lemma~\ref{circ} holds for $\mathcal{A}_U$. If Statement~2 or Statement~3 of Lemma~\ref{circ} holds for $\mathcal{A}_U$ then there exists an $\mathcal{A}_U$-subgroup $A$ of order~$p$. Since $|X^{\pi}|=1$, we conclude that $\mathcal{A}_A=\mathbb{Z}A$. Statement~2 of Lemma~\ref{tenspr} yields that
$$\mathcal{A}_U=\mathcal{A}_C\otimes \mathbb{Z}A~\eqno(8).$$
If Statement~4 of Lemma~\ref{circ} holds for $\mathcal{A}_U$ then each basic set of $\mathcal{A}_U$ outside $C$ is a union of some cosets by a subgroup of $C$. Therefore in this case
$$C_1\leq \rad(X).~\eqno(9)$$

Now suppose that $|X^{\pi}|=p$. The number $\lambda=|Cx\cap X|$ does not depend on $x\in X$. Clearly, $\lambda\in\{1,2,3,4\}$. If $\lambda=4$ then $C\leq \rad(X)$ and hence Eq.~(9) holds for $X$. Put $Y=(X^{[2]})^{[2]}$. If $\lambda\in\{1,3\}$ then $Y=\{x^4:~x\in X\}$. Clearly, $Y\subseteq H$. Since $|X^{\pi}|=p$, we have $|Y|=p$. From Lemma~\ref{sch} it follows that $Y$ is an $\mathcal{A}$-set. So $H=\langle Y \rangle$ is an $\mathcal{A}$-subgroup, a contradiction with the assumption of the proposition. Let $\lambda=2$. Then Eq.~(9) holds for $X$ or for every $x\in X$ the set $Cx\cap X$ contains one element from $C_1$ and one element from $C\setminus C_1$. Put $Z=X^{[2]}$. In the latter case $Z=X_HC_1$. The set $Z$ is an $\mathcal{A}$-set by Lemma~\ref{sch}. Note that $|Z|=|X|$, $Z\cap X\neq \varnothing$, and $Z\neq X$, a contradiction. Thus, we proved that if $|X^{\pi}|=p$ then Eq.~(9) holds for $X$.

If Eq.~(9) holds for every basic set of $\mathcal{A}$ outside $C$ then $\mathcal{A}$ is the $S=C/C_1$-wreath product. Since $|S|=2$, $\mathcal{A}$ and $S$ satisfy Condition~$(\CA)$. So Statement~3 of Theorem~\ref{main2} holds for $\mathcal{A}$. Suppose that there exist $X\in\mathcal{S}(\mathcal{A})$ such that $C_1\nleq \rad(X)$. Then there  exists an $\mathcal{A}$-subgroup $A$  of order $p$ inside $U=\langle X, C \rangle$ such that Eq.~(8) holds for $U$ and $A$. If there exists $T\in \mathcal{S}(\mathcal{A})$ outside $U$ with $C_1\nleq \rad(T)$ then there  exists an $\mathcal{A}$-subgroup $B$  of order $p$ inside $V=\langle T, C \rangle$ such that Eq.~(8) holds for $V$ and $B$. However, this implies that $H=A\times B$ is an $\mathcal{A}$-subgroup, a contradiction with the assumption of the proposition. Therefore Eq.~(9) holds for every basic set of $\mathcal{A}$ outside $U$ and hence $\mathcal{A}$ is the $S=U/C_1$-wreath product. Due to Eq.~(8), we obtain that $\mathcal{A}$ and $S$ satisfy Condition~$(\CA)$. So Statement~3 of Theorem~\ref{main2} holds for $\mathcal{A}$. The proposition is proved.
\end{proof}

\begin{prop}\label{proof3}
In the conditions of  Theorem~\ref{main2}, suppose that $C$ is not an $\mathcal{A}$-subgroup and $X$ is a basic set of $\mathcal{A}$ containing an element of order~$4$. Then one of the following statements holds for $X$:

$(1)$ $X$ is an $L$-coset for some $\mathcal{A}$-subgroup $L$ with $|L|\in\{p,2p,p^2,2p^2\}$;

$(2)$ $X=\langle X \rangle \setminus \rad(X)$. In this case $\langle X \rangle=G$ or $|\rad(X)|\in\{1,2,p\}$.
\end{prop}

\begin{proof}
Without loss of generality, let $c\in X$. Put $L=H\cap Xc^{-1}$. Clearly, $|L|\leq |H|=p^2$. If $|L|\notin\{p,p^2\}$ then $Y=X^{[p]}$ is a subset of $C$ containing $c^p$. Lemma~\ref{sch} implies that $Y$ is an $\mathcal{A}$-set. So $C=\langle Y \rangle$ is an $\mathcal{A}$-subgroup, a contradiction with the assumption of the proposition. Therefore
$$|L|\in\{p,p^2\}.$$

Let us show that $L$ is a subgroup of $H$. If $|L|=p^2$ then $L=H$ and we are done. Let $|L|=p$. For every $k\in \mathbb{Z}$ coprime to $p$ there exists $m\in \mathbb{Z}$ such that $m\equiv 1~\mod~4$ and $m \equiv k~\mod p$. The set $X^{(m)}$ is a basic set by Lema~\ref{burn} and $c^m=c\in X^{(m)}\cap X$. So $X^{(m)}=X$ and hence $L=L^{(k)}$. Since $|L|=p$ and $e\in L$, we conclude that $L$ is a subgroup of $H$.

Suppose that  $X\cap (C_1\times H)=\varnothing$. Then $X=Lc$ whenever $X\neq X^{-1}$ and $X=Lc\cup Lc^{-1}=(C_1\times L)c$ whenever $X=X^{-1}$. So Statement~1 of the proposition holds for $X$.

Now suppose that $X\cap (C_1\times H)\neq \varnothing$. There exists $k\in \mathbb{Z}$ such that $k\equiv 1~\mod~4$ and $k\equiv (-1)~\mod~p$. Note that $c^k=c\in X^{(k)}\cap X$. Since $X^{(k)}\in \mathcal{S}(\mathcal{A})$ by Lemma~\ref{burn}, we conclude that $X^{(k)}=X$. This yields that $x^{-1}=x^k\in X$ for every $x\in X\cap (C_1\times H)$. So $X=X^{-1}$. Therefore $X$ is of the form
$$X=Lc \cup Lc^{-1} \cup X_1c_1 \cup X_0~\eqno(10)$$
for some $X_0,X_1\subseteq H$ with $X_0\cup X_1\neq \varnothing$. Moreover, the set $X_1$ is rational. Indeed, for every $k\in \mathbb{Z}$ coprime to $p$ there exists $m\in \mathbb{Z}$ such that $m\equiv 1~\mod~4$ and $m \equiv k~\mod p$. Then $c=c^m\in X\cap X^{(m)}$. The set $X^{(m)}$ is basic by Lemma~\ref{burn} and hence $X^{(m)}=X$. So $(X_1c_1)^{(m)}=X_1^{(k)}c_1\subseteq X$. Therefore $X_1^{(k)}=X_1$ for every $k$ coprime to $p$.

If $|L|=p^2$ then all conditions of Lemma~\ref{separat} holds for $X$ and $C_0\times H$ and hence $X=\langle X \rangle \setminus \rad(X)=G\setminus \rad(X)$. So Statement~2 of the proposition holds for $X$. 

Let $|L|=p$. Denote the canonical epimorphism from $G$ to $G/\rad(X)$ by~$\pi$. Put $T=(X_0\setminus X_1) \cup (X_1\setminus X_0)$. We divide the rest of the proof into two lemmas deal with the cases $T\neq \varnothing$ and $T=\varnothing$.

\begin{lemm}
In the above notations, suppose that $T\neq \varnothing$. Then Statement~2 of Proposition~\ref{proof3} holds for $X$.
\end{lemm}

\begin{proof}
One can see that $x^2\in X^{[2]}$ for each $x\in T$. Since $|L|=p$, we obtain that $X^{[2]}=T^{(2)}\subseteq H$. In view of Lemma~\ref{sch}, the set $X^{[2]}$ is an $\mathcal{A}$-set and hence $A=\langle X^{[2]} \rangle$ is an $\mathcal{A}$-subgroup inside $H$. Clearly, 
$$\langle T \rangle \subseteq A.~\eqno(11)$$ 
If $X_0\setminus X_1\neq \varnothing$ then $T\cap X\neq \varnothing$. So $A\cap X\neq\varnothing$. This means that $X\subseteq A\leq H$, a contradiction with Eq.~(10). Thus, 
$$T=X_1\setminus X_0.~\eqno(12)$$

Note that $|A|\in \{1,p,p^2\}$ because $A\leq H$. Let us consider all possibilities for $|A|$.
\\
\\
\noindent\textbf{Case 1: $|A|=p^2$.} In this case $A=H$. Due to the condition of the lemma and Eq.~(12), we have $X_1\neq \varnothing$. So in view of Eq.~(10), $\rk(\mathcal{A}_{G/H})=2$. However, $\mathcal{A}_{G/H}$ is a $2$-$S$-ring by Lemma~\ref{minpring}, a contradiction. 
\\
\\
\noindent\textbf{Case 2: $|A|=p$.}  Due to rationality of $X_1$ and Eqs.~(11) and~(12), we conclude that $A^\#\subseteq X_1$. So $|Ac_1\cap X|\geq p-1$. From Lemma~\ref{intersection} it follows that that $|Ac\cap X|=|Ac_1\cap X|\geq p-1$. Therefore $A=L$ by Eq.~(10). This means that $L$ is an $\mathcal{A}$-subgroup. Lemma~\ref{intersection} implies that $|Lx\cap X|=p$ for every $x\in X$ and hence $L\leq \rad(X)$. The group $\rad(X)$ is a subgroup of $H$. Indeed, otherwise $c_1\in \rad(X)$ and hence $Xc_1=X$. This means that $X_1=X_0$ which contradicts to the assumption $T\neq \varnothing$. In view of Eq.~(10) and $\rad(X)\leq H$, we obtain that 
$$\rad(X)=L.~\eqno(13)$$ 

Assume that $X_1\cup X_0\nsubseteq L$. The set $X^{\pi}$, where $\pi:G\rightarrow G/\rad(X)$ is the canonical epimorphism, is a basic set of the $S$-ring $\mathcal{A}_{G/L}$ over a cyclic group $G/L$ of order~$4p$. The following properties hold for $X^{\pi}$: (1) $X^{\pi}\neq (G/L)^\#$; (2) the radical of $X^{\pi}$ is trivial; (3) $X^{\pi}$ contains elements $c^{\pi}$, $(c^{\pi})^{-1}$ of order~$4$ and elements from $(X_1c_1)^{\pi}$ of order~$2p$; (4) $\langle X ^{\pi}\rangle=G/L$. Properties~(2)-(4) of $X^{\pi}$ yield that Statement~2-4 of Lemma~\ref{circ} can not hold for $\mathcal{A}_{G/L}$ and hence Statement~1 of Lemma~\ref{circ} holds  for $\mathcal{A}_{G/L}$, i.e. $\rk(\mathcal{A}_{G/L})=2$. However, this contradicts to Property~(1) of $X^{\pi}$. Thus, $X_1\cup X_0\subseteq L$. Due to Eq.~(13), we conclude that $X_1=L$, $X_2=\varnothing$, and hence $X=Lc\cup Lc^{-1}\cup Lc_1=(C\times L)\setminus L$. So Statement~2 of the proposition holds for $X$.
\\
\\
\noindent\textbf{Case 3: $|A|=1$.} In this case $X_1=X_0\cup \{e\}$. All conditions of Lemma~\ref{separat} holds for $X$ and $C_1$ and hence $X=\langle X \rangle \setminus \rad(X)$ with $\rad(X)\leq C_1$. So Statement~2 of the proposition holds for $X$. The lemma is proved.
\end{proof}

\begin{lemm}
In the above notations, suppose that $T=\varnothing$. Then Statement~2 of Proposition~\ref{proof3} holds for $X$.
\end{lemm}

\begin{proof}
The condition of the lemma implies that $X_1=X_0$ and hence $C_1\leq \rad(X)$. It is easy to see that $\rad(X)\leq C_1\times H$. So due to Eq.~(10), we conclude that 
$$\rad(X)\in\{C_1, C_1\times L\}.$$
Let $\pi:G\rightarrow G/\rad(X)$ be the canonical epimorphism. If $\rad(X)=C_1\times L$ then $X^{\pi}$ is a basic set of an $S$-ring $\mathcal{A}_{G/(C_1\times L)}$ over a cyclic group of order~$2p$. Note  that $X^{\pi}$ contains the unique element $c^{\pi}$ from $H^{\pi}c^{\pi}$ and all elements from $X_0^{\pi}$ which are nontrivial elements from $H^{\pi}$. So neither $C^{\pi}$ nor $H^{\pi}$ is an $\mathcal{A}_{G/(C_1\times H)}$-subgroup. Therefore $\rk(\mathcal{A}_{G/(C_1\times H)})=2$ by Lemma~\ref{circ}. However, in this case $X^{\pi}$ must contain all elements from $H^{\pi}c^{\pi}$, a contradiction. Thus, $\rad(X)=C_1$.

Assume that $X_1\nsubseteq L$. It is easy to see that 
$$X^{\pi}=L^{\pi}c^{\pi}\cup X_1^{\pi}.~\eqno(14)$$ 
Since $X_1\nsubseteq L$, we have $\{(x^{\pi})^2:~x\in X_1\setminus L\}\subseteq (X^{\pi})^{[2]}$. The set $(X^{\pi})^{[2]}$ is an $\mathcal{A}_{G/C_1}$-set by Lemma~\ref{sch}. Clearly, $(X^{\pi})^{[2]}\subseteq H^{\pi}$. So $\langle (X^{\pi})^{[2]} \rangle$ is an $\mathcal{A}_{G/C_1}$-subgroup of $H^{\pi}$ containing elements from $X_1^{\pi}$. Therefore $X^{\pi}\subseteq \langle (X^{\pi})^{[2]} \rangle \leq H^{\pi}$, a contradiction with Eq.~(14). Thus, $X_1\subseteq L$. Since $X_1$ is rational, $X_1=L^\#$. This implies that $X=Lc\cup Lc^{-1} \cup L^\#c_1 \cup L^\#=(C\times L)\setminus C_1$ and Statement~2 of the proposition holds for $X$. The lemma is proved.
\end{proof}

We proved that in any case one of the statements of Proposition~\ref{proof3} holds for $X$. The proposition is proved.
\end{proof}

\begin{prop}\label{proof4}
In the conditions of  Theorem~\ref{main2}, suppose that $C$ is not an $\mathcal{A}$-subgroup. Then one of the statements of Theorem~\ref{main2} holds for $\mathcal{A}$.
\end{prop}

\begin{proof}
Let $X$ be a basic set of $\mathcal{A}$ containing~$c$. Put $U=\langle X \rangle$. Clearly, $U$ is an $\mathcal{A}$-subgroup. From the condition of the proposition it follows that $U\neq C$ and hence
 $$|U|\in \{4p,4p^2\}.$$

Let $|U|=4p^2$, i.e. $U=G$. One of the statements of Proposition~\ref{proof3} holds for $X$. If Statement~1 of Proposition~\ref{proof3} holds for $X$ then 
$$X=Hc~\text{or}~X=(C_1\times H)c$$ 
because $X$ generates $G$. In the former case $X^{-1}=Hc^{-1}$ and $C_1\times H=G\setminus (X\cup X^{-1})$ is an $\mathcal{A}$-subgroup. So $\mathcal{A}$ is the $S=(C_1\times H)/H$-wreath product. Since $|S|=2$, Statement~3 of Theorem~\ref{main2} holds for $\mathcal{A}$. In the latter case $\mathcal{A}=\mathcal{A}_{C_1\times H}\wr \mathcal{A}_{G/(C_1\times H)}$ and hence $\mathcal{A}$ and $S=(C_1\times H)/ (C_1 \times H)$ satisfy Condition~$(\CA)$. Therefore Statement~3 of Theorem~\ref{main2} holds for $\mathcal{A}$. Suppose that Statement~1 of Proposition~\ref{proof3} holds for $X$. Then $X=G\setminus \rad(X)$. If $\rad(X)=e$ then $\rk(\mathcal{A})=2$ and Statement~1 of Theorem~\ref{main2} holds for $\mathcal{A}$; otherwise $\mathcal{A}=\mathcal{A}_{\rad(X)}\wr \mathcal{A}_{G/\rad(X)}$ and $\mathcal{A}$ and $S=\rad(X)/\rad(X)$ satisfy Condition~$(\CA)$. Thus, Statement~3 of Theorem~\ref{main2} holds for $\mathcal{A}$.

Now let $|U|=4p$. The $S$-ring $\mathcal{A}_{G/U}$ is a $p$-$S$-ring by Lemma~\ref{minpring}. So
$$\mathcal{A}_{G/U}\cong \mathbb{Z}C_p.~\eqno(15)$$

\begin{lemm}\label{aux1}
In the above notations, let $Y$ be a basic set of $\mathcal{A}$ outside $U$. Then one of the following statements holds:

$(1)$ $C_1\leq \rad(Y)$;

$(2)$ there exists an $\mathcal{A}$-subgroup $A$ of order~$p$ with $A\nleq U$;

$(3)$ $H$ is an $\mathcal{A}$-subgroup.
\end{lemm}

\begin{proof}
One can write $Y$ in the form
$$Y=Y_1y_1\cup Y_2y_2\cup \ldots Y_sy_s,$$
where $y_i\in H$ and $Y_i\subseteq C$ for every $i\in\{1,\ldots,s\}$. Suppose that $C_1\nleq \rad(Y)$. Then there exists $k\in\{1,\ldots,s\}$ such that $C_1\nleq Y_k$. In this case $|Y_k|\leq 3$. Assume that $|Y_k|=2$. Then $Y_k$ contains one element, say $v$, from $C_1$ and one element, say $w$, from $C\setminus C_1$ because $C_1\nleq \rad(Y)$. There exists $m\in \mathbb{Z}$ such that $m\equiv 3~\mod~4$ and $m\equiv 1~\mod~p$. Then $(y_kv)^m=y_kv\in Y\cap Y^{(m)}$. The set $Y^{(m)}$ is basic by Lemma~\ref{burn}. So $Y^{(m)}=Y$. On the other hand, $(y_kw)^m=y_kw^{-1}\in Y^{(m)}\setminus Y$, a contradiction. Therefore $|Y_k|\in\{1,3\}$. This implies that the set $T=(Y^{[2]})^{[2]}$ contains $y_k^4$. Clearly, $T\subseteq H$. From Lemma~\ref{sch} it follows that $T$ is an $\mathcal{A}$-set and hence $A=\langle T \rangle$ is an $\mathcal{A}$-subgroup with $A\leq H$ and $A\nleq U$. If $|A|=p$ then Statement~2 of the lemma holds; if $A=H$ then Statement~3 of the lemma holds. The lemma is proved. 
\end{proof}

\begin{lemm}\label{aux2}
In the above notations, suppose that $C_1\leq \rad(Y)$ for every $Y\in \mathcal{S}(\mathcal{A})$ outside $U$. Then Statement~3 of Theorem~\ref{main2} holds for $\mathcal{A}$.
\end{lemm}

\begin{proof}
Let $L$ be the least $\mathcal{A}$-subgroup containing $C_1$. Since $C_1\leq U$, we have $L\leq U$. From the assumption of the lemma it follows that $\mathcal{A}$ is the $S=U/L$-wreath product. Note that $L\neq C$ because $C$ is not an $\mathcal{A}$-subgroup. So $|L|\in\{2,2p,4p\}$. If $|L|\in\{2p,4p\}$ then $|S|\leq 2$ and hence $\mathcal{A}$ and $S$ satisfy Condition~$(\CA)$. So Statement~3 of Theorem~\ref{main2} holds for $\mathcal{A}$. 

Let $|L|=2$, i.e. $L=C_1$. If Statement~2 of Proposition~\ref{proof3} holds for $X$ then $\rad(X)=C_1$ and $X=U\setminus C_1$. Therefore 
$$\mathcal{A}=\mathcal{A}_{C_1}\wr \mathcal{A}_{G/C_1}~\eqno(16)$$ 
and $\mathcal{A}$ and $S=C_1/C_1$ satisfy Condition~$(\CA)$. So Statement~3 of Theorem~\ref{main2} holds for $\mathcal{A}$. 

Suppose that Statement~1 of Proposition~\ref{proof3} holds for $X$. Since $|U|=4p$, we have $X=L_1c$ for some $\mathcal{A}$-subgroup $L_1$ with $|L_1|\in\{p,2p\}$. If $|L_1|=2p$ then $L_1\geq C_1$ and $\mathcal{A}$ is the $S_1=L_1/C_1$-wreath product. Since $L_1$ is cyclic and $C_1$ is an $\mathcal{A}_{L_1}$-subgroup, Lemma~\ref{circ} implies that 
$$\mathcal{A}_{L_1}=\mathcal{A}_{C_1}\otimes \mathcal{A}_{L_2}~\text{or}~\mathcal{A}_{L_1}=\mathcal{A}_{C_1}\wr \mathcal{A}_{L_1/C_1},$$ 
where $L_2$ is a subgroup of $L_1$ of order~$p$. In the former case the $S$-rings $\mathcal{A}_{C_1}$ and $\mathcal{A}_{L_2}$ are cyclotomic by Lemma~\ref{circ}. This implies that $\mathcal{A}_{L_1}$ is cyclotomic as the tensor product of two cyclotomic $S$-rings. So $\mathcal{A}_{S_1}$ is also cyclotomic and  $\mathcal{A}_{S_1}$ is Cayley minimal by Lemma~\ref{circcaymin}. Therefore $\mathcal{A}$ and $S_1$ satisfy Condition~$(\CAA)$ and hence Statement~3 of Theorem~\ref{main2} holds for $\mathcal{A}$. In the latter case Eq.~(16) holds for $\mathcal{A}$ and we are done.

Suppose that $|L_1|=p$. Put $U_1=C_1\times L_1$. The group $U_1$ is an $\mathcal{A}$-subgroup. Due to Statement~2 of Lemma~\ref{tenspr}, we obtain that $\mathcal{A}_{U_1}=\mathcal{A}_{C_1}\otimes \mathcal{A}_{L_1}$. Note that $\mathcal{A}_U$ is the $U_1/L_1$-wreath product and $\mathcal{A}_{U/L_1}\cong \mathbb{Z}C_4$ because $X=L_1c$, $X^{-1}=L_1c^{-1}$, and $U_1=U\setminus (X\cup X^{-1})$. Thus, $\mathcal{A}$ and $S$ satisfy Condition~$(\CAAA)$ and Statement~3 of Theorem~\ref{main2} holds for $\mathcal{A}$. The lemma is proved.
\end{proof}

In view of Lemma~\ref{aux2}, we may assume that there exists $Y\in \mathcal{S}(\mathcal{A})$ outside $U$ with $C_1\nleq \rad(Y)$. Then one of the Statements~2-3 of Lemma~\ref{aux1} holds. Suppose that Statement~2 of Lemma~\ref{aux1} holds, i.e. there exists an $\mathcal{A}$-subgroup $A$ of order~$p$ with $A\nleq U$. From Eq.~(15) it follows that $\mathcal{A}_A=\mathbb{Z}A$. Therefore $\mathcal{A}=\mathcal{A}_U \otimes \mathbb{Z}_A$ by Statement~2 of Lemma~\ref{tenspr} and Statement~2 of Theorem~\ref{main2} holds for $\mathcal{A}$.

Suppose that Statement~3 of Lemma~\ref{aux1} holds, i.e. $H$ is an $\mathcal{A}$-subgroup. Then $\mathcal{A}_{G/H}$ is a $2$-$S$-ring by Lemma~\ref{minpring}. So 
$$\mathcal{A}_{G/H}=\mathbb{Z}(G/H)~\text{or}~\mathcal{A}_{G/H}\cong \mathbb{Z}C_2\wr\mathbb{Z}C_2~\eqno(17)$$ by Lemma~\ref{p2}. In both cases $C_1\times H$ is an $\mathcal{A}$-subgroup. If Statement~2 of Proposition~\ref{proof3} holds for $X$ then $X\cap (C_1\times H)\neq \varnothing$ which contradicts to the fact that $C_1\times H$ is an $\mathcal{A}$-subgroup. Therefore Statement~1 of Proposition~\ref{proof3} holds for $X$. Since $|U|=4p$, we have $X=Lc$, where $L$ is an $\mathcal{A}$-subgroup with $|L|\in\{p,2p\}$. Since $H$ is an $\mathcal{A}$-subgroup, $L_1=L\cap H$ is also an $\mathcal{A}$-subgroup. Clearly, $|L_1|=p$. The canonical epimorphism from $G$ to $G/L_1$ is denoted by $\pi$. 

\begin{lemm}\label{aux3}
In the above notations, the $S$-ring $\mathcal{A}_{G/L_1}$ is cyclotomic.
\end{lemm}

\begin{proof}
The $S$-ring $\mathcal{A}_{C^{\pi}}\cong \mathcal{A}_{G/H}$ is a $2$-$S$-ring by Lemma~\ref{minpring}. So $\mathcal{A}_{C^{\pi}}$ is cyclotomic by Lemma~\ref{p2}. Since $|H^{\pi}|=p$, the $S$-ring $\mathcal{A}_{H^{\pi}}$ is cyclotomic by Lemma~\ref{circ}. The group $G/L_1$ is cyclic. Therefore one of the statements of Lemma~\ref{circ} holds for $\mathcal{A}_{G/L_1}$. Statement~1 and Statement~4 of Lemma~\ref{circ} can not hold for $\mathcal{A}_{G/L_1}$ because  $C^{\pi}$ and $H^{\pi}$ are $\mathcal{A}_{G/L_1}$-subgroups. This yields that $\mathcal{A}_{G/L_1}$ is cyclotomic or  $\mathcal{A}_{G/L_1}=\mathcal{A}_{C^{\pi}}\otimes \mathcal{A}_{H^{\pi}}$. In the latter case $\mathcal{A}_{G/L_1}$ is also cyclotomic as the tensor product of two cyclotomic $S$-rings. The lemma is proved.
\end{proof}

Let $Y$ be a basic set of $\mathcal{A}$ outside $U\cup (C_1\times H$). The number $\mu=|Hy\cap Y|$ does not depend on $x\in X$ by Lemma~\ref{intersection}. If $\mu$ is not divisible by~$p$ then $Y^{[p]}\cap(C\setminus C_1)\neq \varnothing$. From Lemma~\ref{sch} it follows that $Y^{[p]}$ is an $\mathcal{A}$-set and hence $C=\langle Y^{[p]} \rangle$ is an $\mathcal{A}$-subgroup, a contradiction with the assumption of the proposition. Therefore $p~|~\mu$ and hence 
$$p~|~|Y|.~\eqno(18)$$

The groups $C^{\pi}$ and $H^{\pi}$ are $\mathcal{A}_{G/L_1}$-subgroups and the set $Y^{\pi}$ lies outside $C^{\pi}\cup H^{\pi}$. The sets $Y^{\pi}_{C^{\pi}}$ and $Y^{\pi}_{H^{\pi}}$ are basic sets of $\mathcal{A}_{G/L_1}$ by Statement~1 of Lemma~\ref{tenspr}. In view of Eq.~(17), we have $Y^{\pi}_{C^{\pi}}\in \{1,2\}$. This yields that $|Y^{\pi}|\in\{|Y^{\pi}_{H^{\pi}}|,2|Y^{\pi}_{H^{\pi}}|\}$. One can see that $|Y^{\pi}_{H^{\pi}}|<p$ because $|H^{\pi}|=p$. Thus,
$$p~\not|~|Y^{\pi}|.~\eqno(19)$$

The number $\nu=|L_1y\cap Y|$ does not depend on $y\in Y$ by Lemma~\ref{intersection}. Also $\nu\leq p$ because $|L_1|=p$. Note that $|Y|=\nu|Y^{\pi}|$. Due to Eqs.~(17) and~(18), we conclude that $p~|~\nu$. Therefore $\nu=p$. This means that $L_1\leq\rad(Y)$. 

The above discussion implies that $L_1\leq\rad(Y)$ for every $Y\in \mathcal{S}(\mathcal{A})$ outside $U\cup (C_1\times H)$. Since also $L_1\leq \rad(X)\cap \rad(X^{-1})$, we conclude that $\mathcal{A}$ is the $S=(C_1\times H)/L_1$-wreath product. The $S$-ring $\mathcal{A}_{G/L_1}$ is cyclotomic by Lemma~\ref{aux3}. So $\mathcal{A}_S$ is also cyclotomic. Note that $S$ is cyclic because $|S|=2p$. Therefore $\mathcal{A}_S$ is Cayley minimal by Lemma~\ref{circcaymin}. Thus, $\mathcal{A}$ and $S$ satisfy Condition~$(\CAA)$ and Statement~3 of Theorem~\ref{main2} holds for $\mathcal{A}$. The proposition is proved.
\end{proof}

\begin{proof}[Proof of Theorem~\ref{main2}]
The statement of the theorem follows from Propositions~\ref{proof1},~\ref{proof2},~\ref{proof4}.
\end{proof}

\section{Proof of Theorem~\ref{main}}
 
As in the previous section, $G$ is a finite group isomorphic to $C_4\times C_p^2$. If $p=2$ then $G$ does not satisfy the necessary condition of  $\DCI$-property from~\cite[Theorem~1.1]{KM} and hence $G$ is not $\DCI$. Let $p$ be odd. In view of Lemma~\ref{cimin}, it is sufficient to prove that $\mathcal{A}=V(K,G)$ is a $\CI$-$S$-ring for every $K \in \Sup_2^{\rm min}(G_\r)$.

\begin{lemm}\label{section}
Let $S$ be an $\mathcal{A}$-section of $G$ and $S\neq G$. Then $\mathcal{A}_S$ is a $\CI$-$S$-ring
\end{lemm}

\begin{proof}
If $S$ is cyclic then the statement of the lemma follows from~\cite[Theorem~3.2]{M2}; if $S$ is not cyclic then the statement of the lemma follows from~\cite[Remark~3.4]{KR2}.
\end{proof}

One of the statements of Theorem~\ref{main2} holds for $\mathcal{A}$. If Statement~1 of Theorem~\ref{main2} holds for $\mathcal{A}$ then, obviously, $\mathcal{A}$ is $\CI$. If Statement~2 of Theorem~\ref{main2} holds for $\mathcal{A}$ then $\mathcal{A}$ is $\CI$ by Lemma~\ref{section} and Lemma~\ref{citens}. If Statement~2 of Theorem~\ref{main2} holds for $\mathcal{A}$ then $\mathcal{A}$ is $\CI$ by Lemma~\ref{section} and one of the Lemmas~\ref{trivial},~\ref{cayleymin},~\ref{restr}. The theorem is proved.


\begin{thebibliography}{list}


\bibitem{Adam}
\emph{A.~\'Ad\'am}, Research Problem 2-10, J. Combin. Theory, \textbf{2} (1967), 393.


\bibitem{AlN}
\emph{B.~Alspach,~L.~Nowitz}, Elementary proofs that $\mathbb{Z}_p^2$ and $\mathbb{Z}_p^3$ are $\CI$-groups, European J. Combin., \textbf{20}, No.~7 (1999), 607--617.


\bibitem{AlPar}
\emph{B.~Alspach, T.~Parsons}, Isomorphism of Circulant graphs and digraphs, Discrete Math., \textbf{25}, No.~2 (1979), 97--108.


\bibitem{BF}
\emph{L.~Babai,~P.~Frankl}, Isomorphisms of Cayley graphs I, Colloq. Math. Soc. J\'anos Bolyai, \textbf{18}, North-Holland, Amsterdam (1978), 35--52.

\bibitem{CP}
\emph{G.~Chen, I.~Ponomarenko}, Lectures on Coherent Configurations (2019), http://www.pdmi.ras.ru/\textasciitilde inp/ccNOTES.pdf.

\bibitem{CLi}
\emph{M.~Conder,~C.~H.~Li}, On isomorphism of finite Cayley graphs, European J. Combin., \textbf{19}, No.~8 (1998), 911--919.



\bibitem{Dob}
\emph{E. Dobson},  Isomorphism problem for Cayley graphs of $\mathbb{Z}_p^3$, Discrete math., \textbf{147}, Nos.~1-3 (1995), 87--94.


\bibitem{ET}
\emph{B.~Elspas,~J.~Turner}, Graphs with circulant adjacency matrices, J. Combin. Theory, \textbf{9} (1970), 297--307.


\bibitem{EP1}
\emph{S.~Evdokimov,~I.~Ponomarenko}, On a family  of Schur rings over a finite cyclic group, St. Petersburg Math. J., \textbf{13}, No.~3 (2002), 441--451. 

\bibitem{EP2}
\emph{S.~Evdokimov,~I.~Ponomarenko}, Schurity of $S$-rings over a cyclic group and generalized wreath
product of permutation groups, St. Petersburg Math. J., \textbf{24}, No.~3 (2013), 431--460. 

\bibitem{EP3}
\emph{S.~Evdokimov,~I.~Ponomarenko}, Coset closure of a circulant S-ring and schurity problem,  J. Algebra Appl., \textbf{15}, No.~4 (2016), 1650068-1-1650068-49.

\bibitem{EKP}
\emph{S.~Evdokimov,~I.~Kov\'acs,~I.~Ponomarenko}, On schurity of finite abelian groups, Comm. Algebra,  \textbf{44}, No.~1 (2016), 101--117.


\bibitem{FK}
\emph{Y-Q.~Feng, I.~Kov\'acs}, Elementary abelian groups of rank~$5$ are $\DCI$-group, J. Combin. Theory Ser. A, \textbf{157} (2018) 162--204.



\bibitem{God}
\emph{C.~Godsil}, On Cayley graph isomorphisms, Ars Combin., \textbf{15} (1983), 231--246.


\bibitem{HM}
\emph{M.~Hirasaka,~M.~Muzychuk}, An elementary abelian group of rank~$4$ is a $\CI$-group, J. Combin. Theory Ser. A, \textbf{94} (2001), 339-362.


\bibitem{KP}
\emph{M.~Klin,~R.~P\"{o}schel}, The K\"{o}nig problem, the isomorphism problem for cyclic graphs and the method of Schur rings, Colloq. Math. Soc. J\'anos Bolyai, \textbf{25} (1981), 405--434.

\bibitem{KM}
\emph{I.~Kov\'acs,~M.~Muzychuk}, The group $\mathbb{Z}_p^2\times \mathbb{Z}_q$ is a $\CI$-group, Comm. Algebra, \textbf{37}, No.~10 (2009), 3500--3515. 

\bibitem{KR1}
\emph{I.~Kov\'acs,~G.Ryabov}, $\CI$-property for decomposable Schur rings over an abelian group, Algebra Colloq., \textbf{26}, No.~1 (2019), 147--160. 

\bibitem{KR2}
\emph{I.~Kov\'acs,~G.Ryabov}, The group $C_p^4 \times C_q$ is a $\DCI$-group, http://arxiv.org/abs/1912.08835 [math.CO], (2019), 1--20.


\bibitem{Li}
\emph{C.~H.~Li}, On isomorphisms of finite Cayley graphs -- a survey, Discrete Math., \textbf{256}, Nos.~1-2 (2002), 301--334.


\bibitem{LiLuP}
\emph{C.~H.~Li, Z.~P.~Lu, P.~P.~P\'alfy}, Further restriction on the structure of finite $\CI$-groups, J. Algebraic Combin. \textbf{26}, No.~2 (2007) 161--181.


\bibitem{M1}
\emph{M.~Muzychuk}, Adam's conjecture is true in the square-free case, J. Combin. Theory Ser. A, \textbf{72} (1995), 118--134.



\bibitem{M2}
\emph{M.~Muzychuk}, On \'Ad\'am's conjecture for circulant graphs, Discrete Math., \textbf{167/168} (1997), 497--510; corrigendum \textbf{176} (1997), 285--298.


\bibitem{MP1}
\emph{M.~Muzychuk,~I.~Ponomarenko}, Schur rings, European J. Combin., \textbf{30}, No.~6 (2009), 1526--1539.


\bibitem{MP2}
\emph{M.~Muzychuk,~I.~Ponomarenko},  On Schur $2$-groups, J. Math. Sci. (N. Y.), \textbf{219}, No.~4 (2016), 565--594.

\bibitem{MS}
\emph{M.~Muzychuk,~G.~Somlai}, The Cayley isomorphism property for $\mathbb{Z}_p^3\times \mathbb{Z}_q$, http://arxiv.org/abs/1907.03570 [math.GR], (2019), 1--13.

\bibitem{NP} 
\emph{R.~Nedela, I.~Ponomarenko},  Recognizing and testing isomorphism of Cayley graphs over an abelian group of order~$4p$ in polynomial time, G. Jones et al (eds). Isomorphisms, Symmetry and Computations in Algebraic Graph Theory, Springer (2020), 195--218.


\bibitem{Ry}
\emph{G.~Ryabov}, On Cayley representations of finite graphs over abelian $p$-groups, Algebra Analiz, \textbf{32}, No.~1 (2020), 94--120.


\bibitem{Schur}
\emph{I.~Schur}, Zur theorie der einfach transitiven Permutationgruppen, S.-B. Preus Akad. Wiss.
Phys.-Math. Kl., \textbf{18}, No.~20 (1933), 598--623.



\bibitem{Wi}
\emph{H.~Wielandt}, Finite permutation groups, Academic Press, New York - London, 1964.


\end{thebibliography}
\end{document}